\documentclass[11pt]{article}
\usepackage{mathrsfs}
\usepackage{latexsym,lineno}
\usepackage{epsfig}
\usepackage{color}
\usepackage{amsmath}\usepackage{fleqn}\usepackage{verbatim}\usepackage{epsf}
\usepackage{amsthm}\usepackage{graphicx, float}\usepackage{graphicx}
\usepackage{amsfonts}\usepackage{amssymb}\usepackage{graphpap}
\usepackage{epic}\usepackage{curves}

\newcommand{\be}{\begin{equation}}
\newcommand{\ee}{\end{equation}}
\newcommand{\benum}{\begin{enumerate}}
\newcommand{\eenum}{\end{enumerate}}
\newcommand{\bit}{\begin{itemize}}
\newcommand{\eit}{\end{itemize}}
\newtheorem{thm}{Theorem}[section]
\newtheorem{lemma}{Lemma}[section]

\newtheorem{prop}{Proposition}[section]

\newtheorem{claim}{Claim}

\usepackage{graphicx}

\begin{document}
\def\s{\subseteq}
\def\n{\noindent}
\def\se{\setminus}
\def\dia{\diamondsuit}
\def\la{\langle}
\def\ra{\rangle}


\title{On the sharp upper and lower bounds of multiplicative Zagreb indices of graphs with connectivity at most $k$}

\author{Shaohui Wang\footnote{  Authors'email address: S. Wang (e-mail:
shaohuiwang@yahoo.com; swang@adelphi.edu).}  \\
\small\emph { Department of Mathematics and Computer Science, Adelphi University, Garden City, NY 11550, USA} }
\date{}
\maketitle

\begin{abstract}

For a (molecular) graph, the first multiplicative Zagreb index $\prod_1(G) $ is the product of the square of every vertex degree, and the second multiplicative Zagreb index $\prod_2(G) $ is the product of the products of degrees of pairs of adjacent vertices.   In this paper, we explore graphs in terms of (edge) connectivity. The maximum and minimum values of $\prod_1(G) $ and $\prod_2(G) $ of graphs with connectivity at most $k$ are provided.  In addition, the corresponding extremal graphs are characterized,
and our results extend and enrich some known conclusions.

\end{abstract}

\section{Introduction}
A topological index is a single number which can be used to describe some properties of  a molecular graph that is  a finite simple graph, representing the carbon-atom skeleton
of an organic molecule of a hydrocarbon.
In recent decades these could be useful for the study of quantitative structure-property relationships (QSPR) and quantitative
structure-activity relationships (QSAR)  and for  the structural essence of biological and chemical compounds.  The well-known Randi$\acute{c}$ index is one of the most important topological indices.

In 1975, Randi$\acute{c}$ introduced a  moleculor quantity of branching index \cite{1}, which has been known as the famous Randi$\acute{c}$ connectivity index that is most useful structural descriptor in QSPR and QSAR, see \cite{2,3,4,5}.  Mathematicians have considerable interests in the structural and applied issues of  Randi$\acute{c}$ connectivity index, see \cite{6,7,8,9}. Based on the successful considerations,  Zagreb indices\cite{10} are introduced as an expected formula for the total $\pi$-electron energy of conjugated molecules as follows.
\begin{eqnarray} \nonumber
M_1(G) = \sum_{u \in V(G)} d(u)^2
~\text{ and}
~ M_2(G) = \sum_{uv \in E(G)} d(u)d(v),
\end{eqnarray}
where $G$ is a (molecular) graph, $uv$ is a bond between two atoms $u$ and $v$, and $d(u)$ (or $d(v)$, respectively) is the number of atoms that are connected with $u$ (or $v$, respectively).
Zagreb indices are employed as molecular descriptors in QSPR and QSAR, see \cite{11,12}.  Recently, Todeschini et al.(2010) \cite{13,14} proposed the
following multiplicative variants of molecular structure descriptors:
\begin{eqnarray} \nonumber
 \prod_1(G) = \prod_{u \in V(G)} d(u)^2 ~
\text{ and}\;\;
 \prod_2(G) = \prod_{uv \in E(G)} d(u)d(v) = \prod_{u \in V(G)}
d(u)^{d(u)}.
\end{eqnarray}
In the interdisplinary of mathemactics, chemistry and physics, it is not surprising that  there are numerous studies of properties of the (multiplicative)Zagreb indices of molecular graphs \cite{1000,1001,1002,1003,1004,1005,1006}. 

In view of these results, researchers are intereted with 
finding upper and lower bounds for multiplicative Zagreb indices of graphs  and characterizing the graphs
in which the maximal (respectively, minimal) index values are attained, respectively. In fact, investigations of the above problems, mathematical and computational properties of Zagreb
indices have also been considered in \cite{16, 17, 18}. Other directions of investigation include studies of relation between multiplicative Zagreb indices and the corresponding invariant of elements of the graph G (vertices, pendent vertices, diameter, maximum degree, girth, cut edge, cut vertex, connectivity, perfect matching). As examples, the first and second multiplicative Zagreb indices for a class of chemical dendrimers are explored by Iranmanesh et al.~\cite{Iranmanesh20102}.
Based on trees, unicyclic
graphs and bicyclic graphs, Borovi\'canin et al.~\cite{Borov2016} introduced the
bounds on Zagreb indices with a fixed  domination number.
 The maximum and minimum
Zagreb indices of trees with given
number of vertices of maximum degree are proposed by  Borovi$\acute{c}$anin and Lampert\cite{Bojana2015}.  Xu and Hua~\cite{Xu20102} introduced an unified approach to
characterize extremal maximal and minimal multiplicative Zagreb
indices, respectively.
Considering the high dimension trees, $k$-trees, Wang and Wei~\cite{Wang2015} provided the maximum and minimum indices of these indices and the corresponding extreme graphs are provided.
 Some sharp upper bounds for
$\prod_1$-index and $\prod_2$-index in terms of graph parameters are investigated by Liu and Zhang \cite{Liuz20102},
including an order, a size and a radius.
Wang et al. \cite{WW} provided  sharp bounds for these indices of of trees with given number of vertices of maximum degree.
  The bounds for
the moments and the probability generating function of these
indices in a randomly chosen molecular graph with tree structure
of order $n$ are studied by Kazemi~\cite{Ramin2016}.  Li and Zhao obtained upper bounds on Zagreb indices of bicyclic graphs with a given matching number~\cite{LL}.

In light of the information available for multiplicative Zagreb indices,
and inspired by  above results, in this paper we further
investigate these indices of graphs with (edge) connectivity. We give some basic
properties of the first and the second multiplicative Zagreb indices.  The maximum and minimum values of $\prod_1(G) $ and $\prod_2(G) $ of graphs with (edge) connectivity at most $k$ are provided.  In addition, the
corresponding extreme graphs are charaterized.
 In our exposition we will use the terminology and notations of (chemical) graph theory(see \cite{BB,NT}). 

\section{Preliminaries}
Let $G$ be a simple  connected  graph, denoted by $G = (V(G), E(G))$, in which $V = V (G)$ is  vertex
set and $E = E(G)$ is  edge set. If a vertex $v \in V(G)$,  then the neighborhood of  $v $ denotes the set $N(v) = N_G(v) = \{w \in V(G), vw \in E(G)\}$, and $d_G(v)$ (or   $d(v)$)  is the degree of $v$ with $ d_{G}(v) = |N(v)|$. $n_i$ is the number of vertices of degree $i \geq 0$. 
For $S \subseteq V(G)$ and  $F \subseteq E(G)$,   we use $G[S]$ for the subgraph of $G$ induced by the vertex set $S$, $G - S$ for the subgraph induced by  $V(G) - S$ and $G - F$ for the subgraph of G obtained by deleting $F$. If $G-S$ contains at least 2 components, then $S$ is said to be a vertex cut set of $G$. Similarly, if $G-F$ contains at least 2 components, then $E$ is called an edge cut set.

A graph $G$ is said to be $k$-connected with $k \geq 1$, if either $G$ is complete graph $K_{k+1}$, or it has at least $k+2$ vertices and contains no $(k-1)$-vertex cut.   
The connectivity of $G$, denoted by $\kappa(G)$, is defined as the maximal value of $k$ for which a connected graph $G$ is $k$-connected. Similarly, 
for $k \geq 1,$ a graph $G$ is called $k$-edge-connected if it has at least two vertices and does not contain an $(k-1)$-edge cut. The maximal value of $k$ for which a connected graph $G$ is $k$-edge-connected is said to be  the edge connectivity
of $G$, denoted by $\kappa'(G)$. According to above definitions, the following proposition is obtained.

\begin{prop} \label{p3} Let $G$ be a graph with $n$ vertices. Then \\
(i) $\kappa(G) \leq \kappa'(G) \leq n-1,$
\\
(ii) $\kappa(G) = n-1, \kappa'(G) = n-1$ and $G \cong K_n$ are equivalent.
\end{prop}
Let $\mathbb{V}_n^k$ be a set of graphs with $n$ vertices and $\kappa(G) \leq k \leq n-1.$ Denote $\mathbb{E}_n^k$ by a set of graphs with $n$ vertices and $\kappa'(G) \leq k \leq n-1.$  For $|V(G)| = n$  and $|E(G)| = n-1$,  $G$ is a tree. Let $P_n$ and $S_n$ be special trees: a path and a star of n vertices. $K_n$ is a complete graph. The graph $K_n^k$ is obtained by joining $k$ vertices of $K_{n-1}$ to an isolated vertex, see Fig 1. Then $K_n^k \in \mathbb{E}_n^k \subset \mathbb{V}_n^k$. 
\begin{figure}[htbp]
    \centering
    \includegraphics[width=4in]{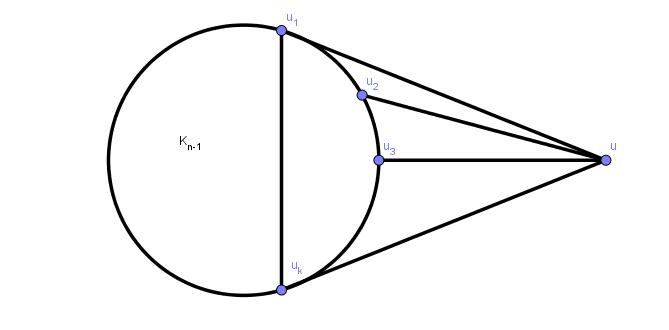}
    \caption{ $K_n^k$.}
    \label{f1}
\end{figure} 

Considering the concepts of $\prod_1(G)$ and $\prod_2(G)$, the  following proposition is routinely obtained.

\begin{prop}\label{p4}
Let $e$ be an edge of a graph $G \in  \mathbb{V}_n^k$ ($\mathbb{E}_n^k$, respectively). Then\\
(i) $G- e \in \mathbb{V}_n^k$ ($\mathbb{E}_n^k$, respectively),\\
(ii) $\prod_i(G-e) < \prod_i({G}).$
\end{prop}

In addition, by elementary calculations, these two statements are deduced.

\begin{prop}\label{p1}
 If $m \geq 0$, then $F_1(x) = \frac{(x+m)^x}{(x-1+m)^{x-1}}$ is an increasing function.
\end{prop}

\begin{prop}\label{p2}
 If $m \geq 0$, then  $F_2(x) = \frac{x^x}{(x+m)^{x+m}}$ is a decreasing function.
\end{prop}

\section{Lemmas and main results}
In this section, the   maximal and minimal multiplicative Zagreb indices of graphs with connectivity at most $k$ in $\mathbb{V}_n^k$ and $\mathbb{E}_n^k$ are determined. The corresponding extremal graphs shall be characterized.
We first provide some lemmas, which are very important and will be used in the proof of our main results.

\begin{lemma}\label{l1}\cite{Iranmanesh20102}
Let $T$ be a tree of n vertices. If $T$ is not $P_n$ or $S_n$, then $\prod_1{(T)} > \prod_1(S_n)$ and $\prod_1(T) > P_n$.
\end{lemma}

Considering the definitions of $\prod_1(G)$ and $\prod_2(G)$, we have the following lemma.
\begin{lemma}\label{l2}
Let $u, v \in V(G)$ and $uv \notin E(G)$. Then $$\prod_1(G+uv) > \prod_1(G)~\text{and}~\prod_2(G+uv) > \prod_1(G).$$
\end{lemma}

Given two graphs $G_1$ and $G_2$, if $V(G_1) \cap V(G_2) = \phi$, then the join graph $G_1 \oplus  G_2$ is a graph with vertex set $V(G_1) \cup V(G_2)$ and edge set $E(G_1) \cup E(G_2) \cup \{uv, u\in V(G_1), v \in V(G_2)\}$. 
\begin{figure}[htbp]
    \centering
    \includegraphics[width=5in]{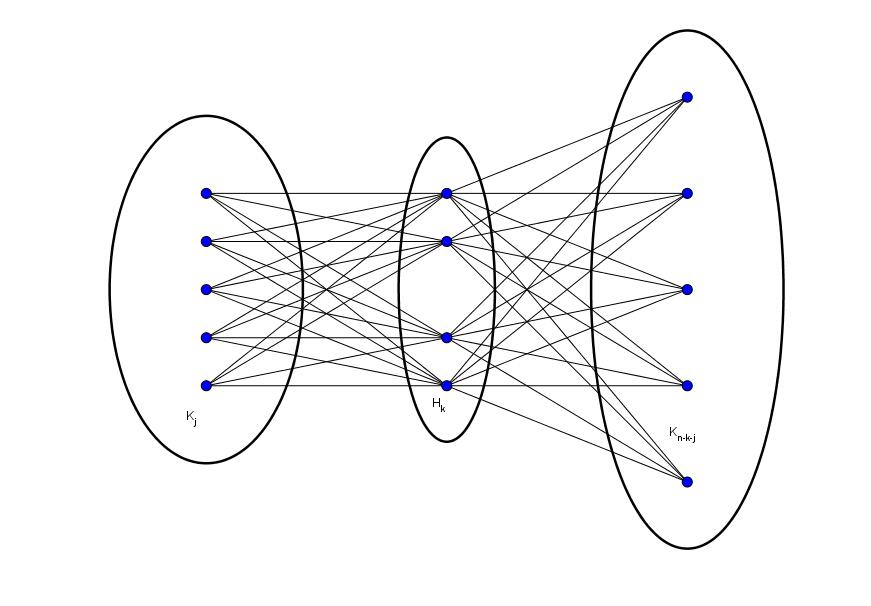}
    \caption{ $G(j,n-k-j) = K_j \oplus H_k \oplus K_{n-k-j}$.}
    \label{f1}
\end{figure} 
\begin{lemma}\label{l3} Let $G(j,n-k-j) = K_j \oplus H_k \oplus K_{n-k-j}$ be a graph with $n$ vertices, in which $K_j$ and $K_{n-k-j}$ are cliques, and $H_k$ is a graph with $k$ vertices, see Fig 2. If $k \geq 1$ and $2 \leq j \leq \frac{n-k}{2}$, then
$$\prod_1(G(j,n-k-j)) < \prod_1(G(1, n-k-1)).$$
\end{lemma}

\begin{proof}
We consider the graph from $G_1= G(j,n-k-j)$ to $G_2 = G(j-1,n-k-j+1)$. Note that if $v \in V(H_k)$ in $G_2$, then $d_{G_2}(v) = d_{G_1}(v)$;  if $v \in V(K_j)$ in $G_2$, then $d_{G_2}(v) = d_{G_1}(v) - 1= j+k-2$;  if $v \in V(K_{n-k-j+1})$ in $G_2$, then $d_{G_2}(v) = d_{G_1}(v)+1 = n-j$.  
By the concepts of $\prod_1$ and $\prod_2$, we have
\begin{eqnarray}
 \nonumber   \frac{\prod_1({G_1})}{\prod_1({G_2})}  
&&=  \frac{\prod_{v \in V(K_j)}d(v)^2 \prod_{v \in V(H_k)}d(v)^2 \prod_{v \in V(K_{n-k-j})}d(v)^2}
{\prod_{v \in V(K_{j-1})}d(v)^2 \prod_{v \in V(H_k)}d(v)^2 \prod_{v \in V(K_{n-k-j+1})}d(v)^2}
 \\  \nonumber 
&&= \frac{\big((j+k-1)^2\big)^j \big((n-j-1)^2 \big)^{n-k-j}}
{\big((j+k-2)^2\big)^{j-1} \big((n-j)^2 \big)^{n-k-j+1}}
  \\ \nonumber    
&&= \bigg(\frac{ \frac{\big(j+(k-1)\big)^j}{\big((j-1) + (k-1)\big)^{j-1} }  }{ \frac{\big((n-j-k+1) + (k -1)\big)^{n-k-j+1}}{\big((n-j-k) + (k-1)\big)^{n-k-j} }}\bigg)^2.
 \end{eqnarray}

Since $2 \leq j \leq \frac{n-k}{2}$, then $j \leq n-k-j < n-k-j+1$. By Proposition \ref{p1} and $k \geq 1$, we have $$\frac{\prod_1({G_1})}{\prod_1({G_2})} < 1,$$ that is, $\prod_1({G_1}) < \prod_1({G_2})$.

We can recursively use this process from $G_1$ to $G_2$, and obtain that 
$$\prod_1(G(j,n-k-j)) < \prod_1(G(j-1,n-k-j+1)) < \prod_1(G(j-2,n-k-j+2)) < \cdots < \prod_1(G(1, n-k-1)).$$
Therefore, $\prod_1(G(j,n-k-j)) < \prod_1(G(1, n-k-1))$. Thus, we complete the proof.
\end{proof}

\begin{lemma}\label{l4}
Let $G$ be a connected graph and $u, v \in V(G)$. Assume that $v_1, v_2, \ldots, v_s \in N(v)\setminus N(u)$, 
$1 \leq s \leq d(v)$. Let $G' = G - \{vv_1, vv_2, \ldots, vv_s \} + \{uv_1, uv_2, \ldots, uv_s\}$. 
If $d(u) \geq d(v)$ and $u$ is not adjacent to $v$, then
   $$\prod_2(G') > \prod_2(G). $$  
\end{lemma}

\begin{proof}
By the concept of $\prod_2(G)$, we have
 \begin{eqnarray}
   \frac{\prod_2(G)}{\prod_2(G')} =  
 \frac{d(u)^{d(u)} d(v)^{d(v)}}{(d(u)+s)^{d(u)+s}(d(v)-s)^{d(v)-s}}   = \frac{\big( \frac{d(u)^{d(u)}}{(d(u)+s)^{d(u)+s}}\big)}{\big( \frac{(d(v)-s)^{d(v)-s}}{d(v)^{d(v)}}\big)}.\nonumber
 \end{eqnarray}

Since $d(u) \geq d(v) > d(v) -s$ and by Proposition \ref{p2}, then  $$\frac{\prod_2(G)}{\prod_2(G')} < 1,$$ that is,
 $ \prod_2(G') > \prod_2(G)$. Thus, the lemma is proved.
\end{proof}

\begin{lemma}\label{l5}
 If $k \geq 1$ and $2 \leq j \leq \frac{n-k}{2}$, we have $$\prod_2(G(j,n-k-j)) < \prod_2(G(1, n-k-1)).$$
\end{lemma}

\begin{proof}
Let $V(K_j) = \{v_1, v_2, \cdots, v_j\}$ and $ V(K_{n-k-j}) = \{u_1, u_2, \cdots, u_{n-k-j}\}$. Note that vertex set $\{ v_2, v_3, \cdots v_j \} \subset N(v_1) \cap V(K_j)$. 
We create a new graph $G' = G(j,n-k-j) - \{v_1v_2, v_1v_3, \ldots, v_1v_j \} + \{u_1v_2, u_1v_3, \ldots, u_1v_j\}$. By $d(v_1) \leq d(u_1)$ and Lemma \ref{l4}, we have $\prod_2(G') \geq \prod_2(G(j,n-k-j))$.

Note that for $G'$, $v_1$ has neighbors in $V(H_k)$ only. Let $G'' = G' + \{v_iu_l, 2 \leq i \leq j, 1 \leq l \leq n-k-j$ and $v_iu_l \notin E(G')\}$.  By Lemma \ref{l2}, we have $\prod_2(G'') > \prod_2(G') \geq \prod_2(G)$.
Therefore $G'' \cong G(1, n-k-1)$ and $\prod_2(G(j,n-k-j)) < \prod_2(G(1, n-k-1))$. Thus the proof is complete.  
\end{proof}

 Next we will turn to prove our main results. In Theorems \ref{t1} and \ref{t2}, the sharp upper bounds of mutiplicative Zagreb indices of graphs in $\mathbb{V}_n^k$ and $\mathbb{E}_n^k$ are proposed.  

\begin{thm}\label{t1}
Let $G$ be a graph in $\mathbb{V}_n^k$. Then
\begin{eqnarray}&&\prod_1(G) \leq k^2 (n-k)^2k (n-2)^{2(n-k-1)}~ \text{and} \nonumber\\ \nonumber
&& \prod_2(G) \leq k^k (n-1)^{k(n-1)} (n-2)^{(n-2)(n-k-1)},
\end{eqnarray}
 where the equalities hold if and only if $G \cong K_n^k$. 
\end{thm}

\begin{proof}
Note that the degree sequence of $K_n^k$ is $k, \underbrace{n-2, n-2, \cdots, n-2}_{n-k-1}, \underbrace{n-1, n-1, \cdots, n-1}_k$. By the concepts of $\prod_1(G)$, $\prod_2(G)$ and routine calculations, we have 
\begin{eqnarray}&&\prod_1(K_n^k) = k^2 (n-1)^{2k} (n-2)^{2(n-k-1)}~ \text{and} \nonumber\\ \nonumber
&& \prod_2(K_n^k) = k^k (n-1)^{k(n-1)} (n-2)^{(n-2)(n-k-1)}.
\end{eqnarray}
It suffices to prove that $\prod_1(G) \leq \prod_1(K_n^k)$ and $\prod_2(G) \leq \prod_2(K_n^k)$, and the equalities hold if and only if $G \cong K_n^k$. 

If $k \geq n-1$, then $G \cong K_n^{n-1} \cong K_n$, and the theorem is true. If $1 \leq k \leq n-2$, then choose a graph $\overline{G}_1$ ($\overline{G}_2$, respectively) in $\mathbb{V}_n^k$  such that $\prod_1(\overline{G}_1)$ ($\prod_2(\overline{G}_2)$, respectively) is maximal.
Since $\overline{G}_i \ncong K_n$ with $i = 1, 2$, then $\overline{G}_i$ has a vertex cut set of size $k$. Let $V_i = \{v_{i1}, v_{i2}, \cdots, v_{ik}\}$ be the cut vertex set of $\overline{G}_i$.  Denoted $\omega(\overline{G}_i - V_i)$ by the number of  components of $\overline{G}_i - V_i$.  In order to prove our theorem, we start with several claims.

\begin{claim}\label{c1}
$\omega({\overline{G}_i - V_i}) = 2$ with $i =1, 2$.
\end{claim} 
\noindent{\small \it Proof.} 
We proceed to prove it by a contradiction. Assume that $\omega({\overline{G}_i - V_i}) \geq 3$ with $i = 1, 2$. Let $G_1, G_2, \cdots, G_{\omega({\overline{G}_i - V_i})}$ be the components of $\overline{G}_i - V_i$. Since $\omega({\overline{G}_i - V_i}) \geq 3$, then choose vertices $u \in V(G_1)$ and $v \in V(G_2)$. Then $V_i$ is still a $k$-vertex cut set of $\overline{G}_i+uv$. By Lemma \ref{l2}, we have $\prod_i(\overline{G}_i+uv) > \prod_i(\overline{G}_i)$, a contradiction to the choice of $\overline{G_i}$. Thus, this claim is proved.

Without loss of generality, suppose that $\overline{G}_i - V_i$ contains only two connected components, denoted by $G_{i1}$ and $G_{i2}$.

\begin{claim}\label{c2}
The induced subgraphs of $V(G_{i1}) \cup V_{i}$ and $V(G_{i2}) \cup V_{i}$  in $\overline{G}_i$ are complete subgraphs.
\end{claim}
\noindent{\small \it Proof.}
We use a contradiction to show it. Suppose that $\overline{G}_i[V(G_{i1}) \cup V_{i}]$  is not a complete subgraph of $\overline{G}_i$. 
Then there exists an edge $uv \notin \overline{G}_i[V(G_{i1}) \cup V_{i}]$. Since $\overline{G}_i[V(G_{i1}) \cup V_{i}] + uv \in \mathbb{V}_n,k$, by Lemma \ref{l2}, we have $\prod_i(\overline{G}_i[V(G_{i1}) \cup V_{i}] +uv) > \prod_i(\overline{G}_i[V(G_{i1}) \cup V_{i}])$, which is a contadiction. Thus, we show this claim.

By the above claims, we see that $G_{i1}$ and  $G_{i2}$ are complete subgraph of $\overline{G}_i$. Let $G_{i1} = K_{n'}$ and  $G_{i2}= K_{n''}$. Then we have $\overline{G}_i = K_{n'} \oplus  \overline{G}_i[V_i] \oplus K_{n''} $. 

\begin{claim}\label{c3}
Either $n'=1$ or $n''=1$.
\end{claim}
\noindent{\small \it Proof.}
On the contrary, assume that $n', n'' \geq 2$. Without loss of generility, $n' \leq n''.$ For $\prod_i(G)$, by Lemmas \ref{l3} and \ref{l5}, we have a new graph $\overline{G'}_i = K_{1} \oplus  \overline{G}_i[V_i] \oplus K_{n-k-1} $ such that $\prod_i(\overline{G'}_i) > \prod_i(\overline{G}_i)$ and $\overline{G'}_i \in \mathbb{V}_n^k$. This is a contradition to the choice of $\overline{G}_i$.  Thus, either $n'=1$ or $n''=1$, and this claim is showed.

By Lemma \ref{l2}, $\prod_i(K_{1} \oplus  K_{|H_k|} \oplus K_{n-k-1}) > \prod_i( K_{1} \oplus  \overline{G}_i[V(H_k)] \oplus K_{n-k-1})$. Since  $\prod_i(K_n^k) = \prod_i(K_{1} \oplus  K_{|V_i|} \oplus K_{n-k-1})$, then  $\prod_i(K_n^k)$ is maximal and this theorem holds.
\end{proof}

Since $K_n^k \in \mathbb{E}_n^k \subset \mathbb{V}_n^k$, then the following result is immediate.
\begin{thm}\label{t2}
Let $G$ be a graph in $\mathbb{E}_n^k$. Then
\begin{eqnarray}&&\prod_1(G) \leq k^2 (n-k)^2k (n-2)^{2(n-k-1)}  \text{and} \nonumber\\ \nonumber
&& \prod_2(G) \leq k^k (n-1)^{k(n-1)} (n-2)^{(n-2)(n-k-1)},
\end{eqnarray}
 where the equalities hold if and only if $G \cong K_n^k$. 
\end{thm}

In the rest of this paper, we consider the minimal mutiplicative Zagreb indices of graphs $G$ in $\mathbb{V}_n^k$ and $\mathbb{E}_n^k$. By Proposition \ref{p4} (ii), $G$ is a tree with $n$ vertices.  By Lemma \ref{l1} and routine calculations, we have 

\begin{thm} \label{t3}
Let $G$ be a graph in $\mathbb{V}_n^k$. Then 
$$\prod_1(G) \geq (n-1)^2~\text{and} ~ \prod_2(G) \geq 4^{n-2},$$ where the equalities hold if and only if $G \cong S_n$ and $G \cong P_n$, respectively.
\end{thm}

Note that  $P_n, S_n \in \mathbb{E}_n^k \subset \mathbb{V}_n^k$, then the following theorem is obvious.

\begin{thm} \label{t4}
Let $G$ be a graph in $\mathbb{E}_n^k$. Then 
$$\prod_1(G) \geq (n-1)^2~\text{and} ~ \prod_2(G) \geq 4^{n-2},$$ where the equalities hold if and only if $G \cong S_n$ and $G \cong P_n$, respectively.
\end{thm}




\begin{thebibliography}{99}
\bibitem{1}
\textcolor{blue}{M. Randi$\acute{c}$, On characterization of molecular branching, J. Amer. Chem. Soc. 97 (1975) 6609-6615.}

\bibitem{2}
\textcolor{blue}{L.B. Kier, L.H. Hall, Molecular Connectivity in Chemistry and Drug Research, Academic Press, New York, 1976.}

\bibitem{3}
\textcolor{blue}{L.B. Kier, L.H. Hall, Molecular Connectivity in Structure-Activity Analysis, Research Studies Press/Wiley, Letchworth, New York, 1986.}

\bibitem{4}
\textcolor{blue}{L. Pogliani, From molecular connectivity indices to semiempirical connectivity terms: Recent trends in graph theoretical descriptors, Chem. Rev. 100
(2000) 3827-3858.}

\bibitem{5}
\textcolor{blue}{R. Todeschini, V. Consonni, Handbook of Molecular Descriptors, Wiley-VCH, Weinheim, 2000.}

\bibitem{6}
\textcolor{blue}{B. Bollob$\acute{a}$s, P. Erd$\ddot{o}$s, Graphs of extremal weights, Ars Combin. 50 (1998) 225-233.}

\bibitem{7}
\textcolor{blue}{B. Bollob$\acute{a}$s, P. Erd$\ddot{o}$s, A. Sarkar, Extremal graphs for weights, Discrete Math. 200 (1999) 5-19.}

\bibitem{8}
\textcolor{blue}{X.L. Li, Y.T. Shi, A survey on the Randić index, MATCH Commun. Math. Comput. Chem. 59 (2008) 127-156.}

\bibitem{9}
\textcolor{blue}{M. Randi$\acute{c}$, The connectivity index 25 years after, J. Mol. Graph. Modell. 20 (2001) 19-35.}

\bibitem{10}
\textcolor{blue}{
 I. Gutman, N. Trinajsti\'c, Graph theory and molecular
orbitals. Total ��-electron energy of alternant hydrocarbons,
Chem. Phys. Lett. 17 (1972) 535-538.}

\bibitem{11}
\textcolor{blue}{
 S.C. Basak, G.D. Grunwald, G.J. Niemi, Use of graph-theoretic geometric molecular descriptors in structure-activity relationships, in: A.T. Balaban (Ed.),
From Chemical Topology to Three-Dimensional Geometry, Plenum Press, New York, 1997, pp. 73-116.}


\bibitem{12}
\textcolor{blue}{
S.C. Basak, B.D. Gute, G.D. Grunwald, A hierarchical approach to the development of QSAR models using topological, geometrical and quantum chemical
parameters, in: J. Devillers, A.T. Balaban (Eds.), Topological Indices and Related Descriptors in QSAR and QSPR, Gordon \& Breach, Amsterdam, 1999,
pp. 675-696.}


\bibitem{13}
\textcolor{blue}{ R. Todeschini, D. Ballabio, V. Consonni, Novel
molecular descriptors based on functions of new vertex degrees,
in: I. Gutman, B. Furtula (Eds.), Novel Molecular Structure
Descriptors Theory and Applications I, Univ. Kragujevac,
Kragujevac, (2010) 72-100.}

\bibitem{14}
\textcolor{blue}{ R. Todeschini, V. Consonni, New local vertex
invariants and molecular descriptors based on functions of the
vertex degrees, MATCH Commun. Math. Comput. Chem. 64 (2010)
359-372.}


\bibitem{1000}
\textcolor{blue}{W. Gao, M. Farahani, M. Husin, S. Wang,  On the edge-version atom-bond connectivity and geometric arithmetic indices of certain
graph operations, Appl. Math.   Comp., 308 (2017) 11-17.}

\bibitem{1001}
\textcolor{blue}{A.M. Yu, F. Tian, On the spectral radius of unicyclic graphs, MATCH Commun. Math. Comput. Chem. 51 (2004) 97-109.}

\bibitem{1002}
\textcolor{blue}{
S.C. Li, H.B. Zhou, On the maximum and minimum Zagreb indices of graphs with connectivity at most $k$, Appl. Math. Lett. 23 (2010) 128-132.}

\bibitem{1003}
\textcolor{blue}{
X.F. Pan, H.Q. Liu, J.M. Xu, Sharp lower bounds for the general Randić index of trees with a given size of matching, MATCH Commun. Math. Comput.
Chem. 54 (2005) 465-480.}

\bibitem{1004}
\textcolor{blue}{Y. Zhai,  J. Liu, S. Wang,
Structure properties of Koch networks based on networks dynamical systems,
Complexity, vol. 2017, 6210878 (2017).}

\bibitem{1005}
\textcolor{blue}{W. Gao,
M. Jamil, A. Javed, M. Farahani, S. Wang, J. Liu
Sharp bounds of the hyper Zagreb index on acyclic, unicylic and bicyclic graphs,
Discrete Dyn. Nat. Soc., vol. 2017, 6079450 (2017).}

\bibitem{1006}
\textcolor{blue}{ J. Estes,  B. Wei, Sharp bounds of the Zagreb indices of $k$-trees. J. Comb. Optim. (2014) 27: 271-291.}

\bibitem{16}
\textcolor{blue}{K. Das, I. Gutman, Some properties of the second Zagreb index, MATCH Commun. Math. Comput. Chem. 52 (2004) 103-112.}

\bibitem{17}
\textcolor{blue}{I. Gutman, K. Das, The first Zagreb index 30 years after, MATCH Commun. Math. Comput. Chem. 50 (2004) 83-92.}

\bibitem{18}
\textcolor{blue}{A.L. Hou, S.C. Li, L.Z. Song, B. Wei, Sharp bounds for Zagreb indices of maximal outerplanar graphs, J. Comb. Optim.  (2011) 22:252-269.}

\bibitem{Iranmanesh20102}
\textcolor{blue}{ A. Iranmanesh, M.A. Hosseinzadeh, I. Gutman, On
multiplicative Zagreb indices of graphs, Iranian J. Math. Chem. 3
(2) (2012) 145-154.}

\bibitem{Borov2016}
\textcolor{blue}{ B. Borovi\'canin, B. Furtula, On extremal
Zagreb indices of trees with given domination number, Appl. Math.
Comput. 279 (2016) 208-218.}

\bibitem{Bojana2015}
\textcolor{blue}{  B. Borovi$\acute{c}$anin, T.A. Lampert, On the maximum and minimum
Zagreb indices of trees with a given
number of vertices of maximum degree, MATCH Commun. Math. Comput. Chem. 74 (2015) 81-96.}

\bibitem{Xu20102}
\textcolor{blue}{ K. Xu, H. Hua, A unified approach to extremal
multiplicative Zagreb indices for trees, unicyclic and bicyclic
graphs, MATCH Commun. Math. Comput. Chem. 68 (2012) 241-256.}

\bibitem{Wang2015}
\textcolor{blue}{S. Wang, B. Wei, Multiplicative Zagreb indices of
$k$-trees, Discrete Appl. Math. 180 (2015) 168-175.}

\bibitem{x001}
\textcolor{blue}{ Y. Zhai, J.-B. Liu, S. Wang, Structure properties of Koch networks based on networks
dynamical systems, Complexity, vol. 2017, 6210878 (2017).}

\bibitem{Liuz20102}
\textcolor{blue}{ J. Liu, Q. Zhang, Sharp upper bounds for
multiplicative Zagreb indices, MATCH Commun. Math. Comput. Chem.
68 (2012) 231-240.}


\bibitem{WW}
\textcolor{blue}{S. Wang,  C. Wang, L. Chen, J. Liu, On extremal multiplicative Zagreb indices of trees with given number
of vertices of maximum degree, Discrete Appl. Math. in press.}

\bibitem{Ramin2016}
\textcolor{blue}{R. Kazemi, Note on the multiplicative Zagreb
indices, Discrete Appl. Math. 198 (2016) 147-154.}



\bibitem{LL}
\textcolor{blue}{ S. Li, Q. Zhao,
Sharp upper bounds on Zagreb indices of bicyclic graphs with a given
matching number, Mathematical and Computer Modelling 54 (2011) 2869–2879.}




\bibitem{BB}
\textcolor{blue}{ B. Bollob$\acute{a}$s, Modern Graph Theory, Springer-Verlag, 1998.}


\bibitem{NT}
\textcolor{blue}{
N. Trinajsti$\acute{c}$, Chemical Graph Theory, CRC Press, 1992.}



\end{thebibliography}
\end{document}